\documentclass{amsart}%
\usepackage{amsfonts}
\usepackage{amsmath}
\usepackage{amssymb}
\usepackage{graphicx}%
\providecommand{\U}[1]{\protect\rule{.1in}{.1in}}
\newtheorem{theorem}{Theorem}
\theoremstyle{plain}

\newtheorem{corollary}{Corollary}

\newtheorem{lemma}{Lemma}

\newtheorem{proposition}{Proposition}
\newtheorem{remark}{Remark}

\numberwithin{equation}{section}
\begin{document}
\title[polynomial martingales]{Markov processes, polynomial martingales and orthogonal polynomials}
\author{Pawe\l \ J. Szab\l owski}
\address{Department of Mathematics and Information Sciences,\\
Warsaw University of Technology\\
ul Koszykowa 75, 00-662 Warsaw, Poland}
\email{pawel.szablowski@gmail.com}
\thanks{The author is deeply grateful to the unknown referee for his useful
suggestions and general remarks and also for his enormous patience in pointing
out 171 linguistic mistakes, punctation errors and typos, All this enabled
substantial improvement of the paper.}
\date{June 2014}

\begin{abstract}
We study general properties for the family of stochastic processes with
polynomial regression property, that is that every conditional moment of the
process is a polynomial. It turns out that then there exists a family of
polynomial martingales $\left\{  M_{n}(X_{t},t)\right\}  _{n\geq1}$ that
contains complete information on the distribution (both marginal and
transitional) of the process.

We specify conditions expressed in terms of $M_{n}^{\prime}s$ under which a
given process has independent increments and further is a Levy process,
contains reversed martingales, is a harness or quadratic harness. We also give
conditions under which some of these martingales are also reversed martingales.

\end{abstract}
\maketitle

\section{Introduction}

We study a subclass of one dimensional Markov processes $\mathbf{X\allowbreak
=\allowbreak}\left(  X_{t}\right)  _{t\in I}$ defined on a finite or infinite
segment that has the property that all its conditional moments of degree say
$n$ are polynomials of degree not exeeding $n$. Poisson, Wiener,
Ornstein--Uhlenbeck processes or more generally $q-$Wiener and $(\alpha$%
,$q)-$OU process (described in detail in \cite{Szab-OU-W} or briefly in
Subsection 2.1 of \cite{SzablPoly}) are examples of such processes. Similar
approach, using polynomials to derive some properties of stochastic processes,
was applied by Schoutens, and Teugels in \cite{SchTeu98} to study L\'{e}vy
processes. Our approach is general, hence applicable to all Markov processes
that have marginal distributions identifiable by moments.

To be more specific let us assume the following:

Let $\mathbf{X\allowbreak=\allowbreak}(X_{t})_{t\in I}$ be a real stochastic
process defined on some probability space $(\Omega$,$\mathcal{F}$,$P)$ where
$I\allowbreak=\allowbreak\lbrack l$,$r]$ is some finite or infinite segment of
a real line. Cases $l\allowbreak=\allowbreak-\infty$ or $r\allowbreak
=\allowbreak\infty$ are allowed. Let us also assume that $\forall t\in I$ :
$\operatorname*{supp}X_{t}$ contains infinite number of points and that
$\exists\alpha>0$,$\forall t\in I:\mathbb{E}\exp(\alpha\left\vert
X_{t}\right\vert )<\infty$. It is known that then, the marginal measure $\mu$
is identifiable by moments. It turns out that there exist slightly milder
conditions assuring this. For details see e.g. \cite{Akhizer65},
\cite{Chih79}, \cite{Sim05}, \cite{Sim98}.

Let us denote also \newline%
\[
\mathcal{F}_{\leq s}\allowbreak=\allowbreak\sigma(X_{t}:t\in\lbrack
l,s]),\mathcal{F}_{\geq s}\allowbreak=\allowbreak\sigma(X_{t}:t\in\lbrack
s,r])
\]
and
\[
\mathcal{F}_{s,t}\allowbreak=\allowbreak\sigma\left(  X_{v}:v\notin(s,t),v\in
I\right)  .
\]

Moreover, let us assume that $\exists N:\forall0<$ $n\leq N;$ $s_{i}\in I$,
$s_{i}\neq s_{j}$, for $i\neq j$ and $i$, $j\allowbreak=\allowbreak1,\ldots
,n$, the matrix $[\operatorname*{cov}(X_{s_{i}}$,$X_{s_{j}})]_{i,j=1,\ldots
,n}$ is non-singular. Processes satisfying these assumptions will be called
totally linearly dependent of degree $N$ (briefly $N-$TLD).

Notice that processes which for every $t\in I$ are constant i.e.
$X_{t}\allowbreak=\allowbreak X$ for some random variable $X$ are also not TLD.

We will also assume that $\exists N:\forall m$, $j\leq N:\mathbb{E}X_{t}%
^{m}X_{s}^{j}$ are continuous functions of $t$, $s\in I$ at least for $s$
$\allowbreak=\allowbreak t$. Such processes will be called mean-square
continuous of degree $N$ (briefly $N-$MSC).

Let us remark that sequence of independent random variables indexed by some
discrete linearly ordered set are not MSC.

By $L_{2}(t)$ let us denote a space spanned by real square integrable
functions with respect to one-dimensional distribution of $X_{t}$. By our
assumptions in $L_{2}(t)$ there exists a set of orthogonal polynomials that
constitute base of this space.

Thus, the class of Markov processes that we will consider is a class of
stochastic processes that are $N-$TLD and $N-$MSC and moreover, satisfying the
following condition:%

\begin{equation}
\exists N:\forall N\geq n\geq1,s\leq t:\mathbb{E}(X_{t}^{n}|\mathcal{F}_{\leq
s})\allowbreak=\allowbreak R_{n}(X_{s},s,t), \label{p_reg}%
\end{equation}
where $R_{n}(x,s,t)$ is a polynomial of degree not exceeding $n$ in $x$. We
will call this class of processes Markov processes with polynomial regression
of degree $N$ (briefly $N-$MPR process). If $N$ can be taken $\infty$ then, we
will talk of MPR processes. More precisely we should call this class $N-$rMPR
class i.e. right Markov processes with polynomial regression. However, until
we will consider left (with the obvious meaning) class of Markov processes we
will use the name MPR class.

We will study general properties of Markov processes from MPR class.

As shown in \cite{SzablPoly} then, under our assumptions there exists a family
of polynomial martingales $\left\{  M_{n}(X_{t},t)\right\}  _{n\geq0}$ with
$M_{0}(X_{t}$,$t)\allowbreak=\allowbreak1$. It should be underlined that in
the two akin papers written by the author \cite{SzablPoly} and \cite{SzabStac}
there are some additional assumptions based on which the results of these
papers are obtained, e.g. in \cite{SzablPoly} we assume additionally that
polynomial martingales $\left\{  M_{n}\right\}  $ constitute also the family
of orthogonal polynomials. In \cite{SzabStac} we additionally assume that the
analyzed process is stationary.

In this paper we study the most general case. The only assumptions we are
making are those mentioned above. The additional assumptions are imposed to
obtain some additional properties. In particular, by imposing certain
conditions on these polynomial martingales we characterize processes with
independent increments (among them L\'{e}vy processes), harnesses, and
quadratic harnesses. We also specify conditions for a martingale $M_{n}(X_{s}
$,$s)$ to be a reversed martingale or for the set of polynomial martingales
$\left\{  M_{n}(X_{t},t)\right\}  _{n\geq0}$ to be the set of orthogonal martingales.

To see the role of these polynomials let us consider the space $L(\Omega
,\mathcal{F},\mu)$ of measurable real functions such that $\int\exp
(\alpha|f|)d\mu$ is finite for some $\alpha>0$ for all $f\in L$. Then, the
moments i.e. the numbers $\int f^{n}d\mu$ for $n\geq0$ do exist and moreover,
their values enable us to regain the so-called distribution of the function
$f$ i.e. the function $F_{f}(t)\allowbreak=\allowbreak\mu\left\{
\omega:f\left(  \omega\right)  \leq t\right\}  $ for all $t$. In other words
the knowledge of polynomials $\left\{  p_{n}(x)\right\}  $ in the space $L$
such that degree of $p_{n}$ is $n$ and $\int p_{n}(x)d\mu(x)\allowbreak
=\allowbreak0$ allow firstly to build (using the so-called Gram--Schmidt
procedure ) a family of orthogonal polynomials and secondly find the measure
that makes these polynomials orthogonal. The algorithm is rather complicated.
For details see e.g. the excellent monographs \cite{Akhizer65}, \cite{Chih79},
\cite{Sim05} or \cite{Sim98}.

Hence, taking the above into account, the family of polynomial martingales
$\left\{  M_{n}(X_{t},t)\right\}  _{n\geq0}$ contains complete information
about the so-called transitional distribution. More precisely, given
polynomials $p_{n}(x,y,s,t)\allowbreak=\allowbreak M_{n}(x,t)-M_{n}(y,s)$,for
$s<t$, $n\geq1$ we can (by the so-called the Gram--Schmidt procedure) get a
family of polynomials orthogonal with respect to the so-called transitional
measure $\mu(A|s,y,t)$ defined by $\mu(A|s,y,t)\allowbreak=\allowbreak
P(X_{t}\in A|X_{s}\allowbreak=\allowbreak y)$ and consequently, characterize
this transitional measure completely.

Further, we will assume that $\forall n\geq1$,$~\mathbb{E}M_{n}(X_{t}%
$,$t)\allowbreak=\allowbreak0$, hence we get linearly independent family of
polynomials from which (again by Gram--Schmidt procedure) we can get family of
polynomials that are orthogonal with respect to marginal distribution of the
process at time $t$.

In other words, a family of polynomial martingales contains complete
information on the stochastic process that generated these martingales.

It is an interesting question if polynomials $\left\{  M_{n}\right\}  $ can be
made identical with the family of orthogonal polynomials. We will show in the
sequel that it is not always possible.

The paper is organized as follows. In the next section we study general
properties of these polynomial martingales i.e. we study what conditions have
to be imposed on them to get process with independent increments, a reversed
martingale or make them orthogonal. In Section \ref{HAR} we study necessary
conditions for the process $\mathbf{X}$ to be a harness and we also specify
sufficient conditions for this to happen. In Subsection \ref{QHAR} we study
necessary conditions to guarantee that the process is a quadratic harness.
Longer proofs are shifted to Section \ref{dow}.

\section{General properties\label{gen}}

Let us assume that we deal with the Markov process $\mathbf{X\allowbreak
=\allowbreak}\left(  X_{t}\right)  _{t\in I}$ defined on a certain probability
space $(\Omega,\mathcal{F},P)$ with real values that has the property that its
every conditional moment of degree $n$ conditioned upon the past is a
polynomial of degree $n$. $I$ denotes an index set, usually $\mathbb{N\cup
}\left\{  0\right\}  $, $\mathbb{R}^{+}$, $\mathbb{Z}$, $\mathbb{R}$. What is
important about $I$ is that there exists a total order in it i.e. the binary
relation that is antisymmetric, transitive and total.

As it follows from \cite{SzablPoly} the property of being MPR process is
equivalent to the fact that there exist a family of polynomial martingales
$\left\{  M_{n}\left(  X_{t},t\right)  \right\}  _{n\geq0,t\in I}$ that is
\[
\mathbb{E}\left(  M_{n}(X_{t},t\right)  |\mathcal{F}_{\leq s})\allowbreak
=\allowbreak M_{n}(X_{s},s),
\]
for all $n\geq0$ and $s<t$. For the sake of brevity of notation we will
suppress dependence of $M_{n}$ on $X_{t}$ and will abbreviate $M_{n}(X_{s},s)$
to $M_{n}(s)$.

Let us also write $m_{n}(s)\allowbreak$ for $\allowbreak\mathbb{E}M_{n}%
^{2}(s)$. From the theory of martingales it follows that the functions
$\left\{  m_{n}\right\}  _{n\geq0}$ are nondecreasing. We will additionally
assume that all functions $m_{n}(s)$ are continuous and that $\lim
_{s\longrightarrow l}m\left(  s\right)  \allowbreak=\allowbreak0$, where $l$
denotes the left hand side boundary of the index set $I$. Hence, from these
two additional assumptions it follows that $\mathbb{E}M_{n}(t)\allowbreak
=\allowbreak0$.

\begin{proposition}
$\forall n>0$,$s\in I:M_{n}^{2}(s)\allowbreak=\allowbreak m_{n}(s)\allowbreak
+\allowbreak\sum_{j=1}^{2n}c_{j}(s)M_{j}(s)$, for some continuous functions
$c_{j}(s)$, $j\allowbreak=\allowbreak1$,$\ldots$,$2n$.
\end{proposition}

\begin{proof}
It follows the fact that martingales $M_{n}$ are polynomials such that $M_{n}
$ is a polynomial of degree $n$. Hence, is $M_{n}^{2}$ is a polynomial of
degree $2n$ and the result yields by uniqueness of the polynomial expansion.
\end{proof}

Notice that from the listed above assumptions it follows that for $\forall
t\in I$, the moment generating function of $X_{t}$ exists in some neighborhood
of $0$. Consequently we deduce that every measurable function $g(x)$ such that
$\mathbb{E}\left\vert g(X_{t})\right\vert ^{2}<\infty$ can be expanded in a
Fourier series of polynomials that are orthogonal with respect to marginal
measure of $X_{t}$.

By \cite{SzablPoly} we know that the considered class of random processes
contains processes with independent increments (II-class) with all moments
existing. Below we have characterization of II class within the MPR class of
random processes.

\begin{proposition}
\label{nzal}The MPR process has independent increments iff there exist a
sequence of continuos functions $\left\{  g_{k}(t)\right\}  _{k\geq1}$ defined
on $I$ such that martingales $\left\{  M_{n}\left(  t\right)  \right\}
_{n\geq1}$ satisfy the following recurrence:%
\begin{equation}
M_{n}(t)\allowbreak=\allowbreak X_{t}^{n}-g_{n}(t)-\sum_{j=1}^{n-1}\binom
{n}{j}g_{j}(t)M_{n-j}(t), \label{ind}%
\end{equation}
for $n\geq1$. Then, $g_{n}(t)\allowbreak=\allowbreak\mathbb{E}X_{t}^{n}$.
\end{proposition}

\begin{proof}
Since $\mathbb{E}M_{n}(t)\allowbreak=\allowbreak0$, we easily see that
$g_{n}(t)\allowbreak=\allowbreak\mathbb{E}X_{t}^{n}$. Now considering the
lower triangular matrix $V_{n}(t)\allowbreak=\allowbreak\lbrack\binom{i}%
{j}g_{i-j}(t)]_{i,j=0,\ldots,n}$, vectors $\mathbf{X}_{n}\allowbreak
=\allowbreak(1,X_{t},X_{t}^{2},\ldots,X_{t}^{n})^{T}$ and $\mathbf{M}%
_{n}(t)\allowbreak=\allowbreak(1,M_{1}(t),\ldots,M_{n}(t))^{T}$, recurrences
(\ref{ind}) can be expressed in the form $V_{n}(t)\mathbf{M}_{n}%
(t)\allowbreak=\allowbreak\mathbf{X}_{n}$ for all natural $n$. Consequently,
referring to results of \cite{SzablPoly} we deduce that $V_{n}(t)$ is a
structural matrix of the process $\mathbf{X}$\textbf{. }Following the lines of
reasoning presented in \cite{SzablPoly} we deduce that process $\mathbf{X}$
has independent increments.

Conversely, if $\mathbf{X}$ has independent increments then, (\ref{ind}) is
read in the following way
\[
X_{t}^{n}\allowbreak=\allowbreak\sum_{j=0}^{n}\binom{n}{j}g_{n-j}%
(t)M_{j}(t),~\mathbb{E(}X_{t}^{n}|\mathcal{F}_{\leq s})=\sum_{j=0}^{n}%
\binom{n}{j}g_{n-j}(t)M_{j}(s),
\]
and more importantly $\mathbb{E(}(X_{t}-X_{s})^{n}|\mathcal{F}_{\leq s})$ is a
nonrandom quantity, say, equal to $\gamma_{n}(s$,$t)$. On the other hand,
since $\gamma_{n}(s$,$t)\allowbreak=\allowbreak\mathbb{E}(\gamma_{n}(s$,$t))$
we deduce that $\gamma_{n}(s$,$t)\allowbreak=\allowbreak\sum_{j=0}^{n}%
(-1)^{j}\binom{n}{j}\mathbb{E}(X_{t}^{n-j})\mathbb{E}(X_{s}^{j})\allowbreak
=\allowbreak\sum_{j=0}^{n}(-1)^{j}\binom{n}{j}g_{n-j}(t)g_{j}(s)$. Besides,
notice that $\gamma_{n}(s$,$s)\allowbreak=\allowbreak0$. We have%
\begin{gather*}
\mathbb{E}(X_{t}^{n}|\mathcal{F}_{\leq s})\allowbreak=\allowbreak
\mathbb{E(}(X_{t}-X_{s}+X_{s})^{n}|\mathcal{F}_{\leq s})\allowbreak
=\allowbreak\\
=\sum_{j=0}^{n}\binom{n}{j}X_{s}^{j}\gamma_{n-j}(s,t)\allowbreak
=\allowbreak\sum_{j=0}^{n}\binom{n}{j}\gamma_{n-j}(s,t)\sum_{k=0}^{j}\binom
{j}{k}g_{j-k}(s)M_{k}(s)\\
=\sum_{k=0}^{n}\binom{n}{k}M_{k}(s)\sum_{j=k}^{n}\binom{n-k}{j-k}\gamma
_{n-j}(s,t)g_{j-k}(s)\\
=\sum_{k=0}^{n}\binom{n}{k}M_{k}(s)\sum_{l=0}^{n-k}\binom{n-k}{l}%
g_{n-k-l}(s)\sum_{r=0}^{l}(-1)^{r}g_{r}(s)g_{l-r}(t).
\end{gather*}
Now notice that
\begin{align*}
&  \sum_{j=0}^{n}\binom{n}{j}g_{n-j}(s)\sum_{k=0}^{j}(-1)^{j-k}\binom{j}%
{k}g_{k}(t)g_{j-k}(s)\allowbreak\\
&  =\newline\allowbreak\sum_{k=0}^{n}\binom{n}{k}g_{k}(t)\sum_{j=k}%
^{n}(-1)^{j-k}\binom{n-k}{j-k}g_{j-k}(s)g_{n-k-j}(s)\allowbreak=\allowbreak
g_{n}(t),
\end{align*}
$\ $by the properties of $\gamma_{n}(s$,$t)$. Hence, indeed
\[
\mathbb{E}(X_{t}^{n}|\mathcal{F}_{\leq s})\allowbreak\allowbreak
=\allowbreak\sum_{k=0}^{n}\binom{n}{k}M_{k}(s)g_{n-j}(t).
\]

\end{proof}

\begin{corollary}
The MPR process $\mathbf{X}$ is a L\'{e}vy process iff the polynomial
martingales $M_{n}(t)$ satisfy (\ref{ind}) and moreover, the functions
$g_{i}(t)$ satisfy the following condition:
\[
g_{n}(t-s)\allowbreak=\allowbreak\sum_{j=0}^{n}(-1)^{j}\binom{n}{j}%
g_{n-j}(t)g_{j}(s).
\]

\end{corollary}

\begin{proof}
By Proposition \ref{nzal} we know that $\mathbf{X}$ is a process with
independent increments with moment functions $g_{n}(t)\allowbreak
=\allowbreak\mathbb{E}X_{t}^{n}$. Now recall that L\'{e}vy processes are
processes with independent increments satisfying the property that
$X_{t}-X_{s}$ has the same distribution as $X_{t-s}$ for all $t>s$
\end{proof}

In the sequel it will be understood that $t\in I$ often without specifying so.
To proceed further, we need the following observation:

\begin{lemma}
\label{aux}Suppose $s\leq u_{1}\leq$,$\ldots$,$\leq u_{n}$. Then,
\[
\mathbb{E(}M_{k_{1}}(u_{1})\ldots M_{k_{n}}(u_{n})|\mathcal{F}_{\leq
s})\allowbreak=\allowbreak\sum_{j=0}^{k_{1}+\ldots+k_{n}}\phi_{j}(u_{1}%
,\ldots,u_{n-1})M_{j}(s),
\]
for some continuous functions $\phi_{j}(u_{1}$,$\ldots$,$u_{n-1})$.
\end{lemma}

\begin{proof}
The proof is by induction. Let us take $n\allowbreak=\allowbreak2$. Then,
$\mathbb{E(}M_{k_{1}}(u_{1})M_{k_{2}}(u_{2})|\mathcal{F}_{\leq s}%
)\allowbreak=\allowbreak\mathbb{E(}M_{k_{1}}(u_{1})M_{k_{2}}(u_{1}%
)|\mathcal{F}_{\leq s})$ since $M_{k_{2}}$ is a martingale. Further, since
$M_{k_{1}}$ and $M_{k_{2}}$ are polynomials we deduce that $M_{k_{1}}%
(u_{1})M_{k_{2}}(u_{1})\allowbreak=\allowbreak\sum_{j=0}^{k_{1}+k_{2}}\phi
_{j}(u_{1})M_{j}(u_{1}) $ by our assumptions concerning polynomials $M_{i}$.
Hence, now we have
\begin{gather*}
\mathbb{E(}M_{k_{1}}(u_{1})\ldots M_{k_{n}}(u_{n})|\mathcal{F}_{\leq
s})\allowbreak=\mathbb{E(}M_{k_{1}}(u_{1})\mathbb{E(}M_{k_{2}}(u_{2})\ldots
M_{k_{n}}(u_{n}))|\mathcal{F}_{\leq u_{1}})|\mathcal{F}_{\leq s})\\
=\allowbreak\sum_{j=0}^{k_{2}+\ldots+k_{n-1}}\phi_{j}(u_{2},\ldots
,u_{n-1})\mathbb{E(}M_{k_{1}}(u_{1})M_{j}(u_{1})|\mathcal{F}_{\leq
s})\allowbreak\\
=\allowbreak\allowbreak\sum_{j=0}^{k_{1}+\ldots+k_{n-1}}\phi_{j}(u_{2}%
,\ldots,u_{n-1})\sum_{m=0}^{j+k_{n}}\phi_{m}(u_{1})M_{m}(s)\\
\mathbb{=}\sum_{m=0}^{k_{n}}M_{m}(s)\sum_{j=0}^{k_{1}+\ldots+k_{n}}\phi
_{m}(u_{1})\phi_{j}(u_{2},\ldots,u_{n-1})\\
+\sum_{m=k_{n}+1}^{k_{1}+\ldots+k_{n}}M_{m}(s)\sum_{j=m-k_{n}}^{k_{1}%
+\ldots+k_{n}}\phi_{m}(u_{1})\phi_{j}(u_{2},\ldots,u_{n-1}).
\end{gather*}

\end{proof}

Let us first deal with the case when these martingales are also reversed martingales.

\begin{proposition}
\label{wst}Suppose that there exists a function $a(s)$ such that
\begin{equation}
a(s)\mathbb{E}\left(  M_{n}(s\right)  |\mathcal{F}_{\geq t})\allowbreak
=\allowbreak a(t)M_{n}(t), \label{rew}%
\end{equation}
for some fixed $n>0$. Then, one can select $a(s)$ to be positive for all $s$
and $a(s)\allowbreak=\allowbreak1/m_{n}(s)$. Moreover,%

\begin{equation}
\forall j\geq0:\mathbb{E}M_{n}(s)M_{j}(s)\allowbreak=\allowbreak\chi
_{j,n}m_{n}(s), \label{rwv_n}%
\end{equation}
for some constants $\chi_{j,n}$. Conversely, if condition (\ref{rwv_n}) holds
then, $\left\{  M_{n}(s)/m_{n}(s)\right\}  $ is a reversed martingale.
\end{proposition}

\begin{proof}
Multiplying both sides of (\ref{rew}) by $M_{j}(t)$ and taking expectations we
get on the left hand side: $a(s)\mathbb{E}(\mathbb{E}\left(  M_{n}(s\right)
|\mathcal{F}_{\geq t})\allowbreak M_{j}(t))\allowbreak=\allowbreak
a(s)\mathbb{E}M_{n}(s)M_{j}(s)$, since $M_{j}$ is a martingale. On the right
hand side we have: $a\left(  t\right)  \mathbb{E}M_{n}(t)M_{j}(t)$. First
taking $j\allowbreak=\allowbreak n$ we get
\[
a(s)m_{n}(s)\allowbreak=\allowbreak a(t)m_{n}(t),
\]
from which follows that $a(s)m_{n}(s)$ does not depend on $s$. Moreover, since
the functions $m_{n}(t)$ are positive we deduce that the function $a(s) $ does
not change its sign hence it can be selected to be positive and we get the
first assertion. From the equality $\mathbb{E}M_{n}(s)M_{m}(s)/m_{n}%
(s)\allowbreak=\allowbreak\mathbb{E}M_{n}(t)M_{m}(t)/m_{n}(t)$ we get the second.

Now let us assume that (\ref{rwv_n}) holds. To show that $M_{n}(s)/m_{n}(s)$
is a reversed martingale we have to show that for all bounded functions $g$
measurable with respect to $\mathcal{F}_{\geq t}$ we have $\mathbb{E}%
gM_{n}(s)/m_{n}(s)\allowbreak=\allowbreak\mathbb{E}gM_{n}(t)/m_{n}(t)$. Lets
consider $g(u_{1}$,$\ldots$,$u_{n})\allowbreak=\allowbreak\prod_{j=1}%
^{n}M_{k_{j}}(u_{j})$, for $t\leq u_{1}\leq,\ldots,\leq u_{n}$. Then, we have
that
\[
\mathbb{E}M_{n}(s)g(u_{1},\ldots,u_{n})\allowbreak=\allowbreak m_{n}%
(s)\sum_{j=1}^{k_{1}+\ldots+k_{n}}\phi_{j}(u_{1},\ldots,u_{n-1}),
\]
by Lemma \ref{aux} and by our assumption. Similarly we show that:%
\[
\mathbb{E}M_{n}(t)g(u_{1},\ldots,u_{n})\allowbreak=\allowbreak m_{n}%
(t)\sum_{j=1}^{k_{1}+\ldots+k_{n}}\phi_{j}(u_{1},\ldots,u_{n-1}).
\]

Now notice that linear combinations of such functions $g$ are dense in the
space of functions measurable with respect to $\mathcal{F}_{\geq t}$ we deduce
by Dynkin's $\pi-\lambda$ Theorem that this is true for all functions $g$
measurable with respect to $\mathcal{F}_{\geq t}$.
\end{proof}

We have the following theorem

\begin{theorem}
\label{ort}Suppose that all functions $m_{n}(s)$ are different that is
$m_{n}(s)/m_{j}(s)\neq1$ for all $j\neq n$ and all $s$. Then, the following
statements are equivalent:

1. All martingales $\left\{  M_{n}(t)\right\}  _{n\geq1}$ are orthogonal with
respect to the marginal measure of $X_{t}$.

2. There exist continuous functions $a_{n}(t)$ such that $\left\{
a_{n}(s)M_{n}(s)\right\}  _{n\geq0}$ are reversed martingales.
\end{theorem}

\begin{proof}
First let us show that 2. implies 1. From Proposition \ref{wst} we have:%
\[
a_{n}(s)\mathbb{E}M_{n}(s)M_{j}(s)\allowbreak=\allowbreak a_{n}(t)\mathbb{E}%
M_{n}(t)M_{j}(t).
\]
However, the fact that all $a_{n}(s)M_{n}(s)$ are reversed martingales yields
that we have also:
\[
a_{j}(s)\mathbb{E}M_{n}(s)M_{j}(s)\allowbreak=\allowbreak a_{j}(t)\mathbb{E}%
M_{n}(t)M_{j}(t).
\]
Consequently, we deduce that $\mathbb{E}M_{n}(s)M_{j}(s)\allowbreak
=\allowbreak\chi_{j,n}m_{n}(s)\allowbreak=\allowbreak\chi_{n,j}m_{j}(s)$.
Since by assumption all functions $m_{n}(s)$ are different, we deduce that
$\mathbb{E}M_{n}(s)M_{m}(s)\allowbreak\allowbreak=\allowbreak0$ for all $m\neq
n$.

Hence, now let us assume that the polynomials $\left\{  M_{n}\right\}  $ are
orthogonal. Let us fix $n$ then, from the orthogonality assumption we have
$\forall j\geq0:$ $\mathbb{E}M_{n}(t)M_{j}\allowbreak(t)=\allowbreak0m_{n}%
(t)$. By Proposition \ref{wst} we deduce that $M_{n}$ is a reversed martingale.
\end{proof}

\begin{remark}
Following results of \cite{SzabLev} and \cite{SzabStac} one can state that
there exists MPR processes such that under some quite natural assumptions all
polynomial martingales $M_{n}$ can be selected to be orthogonal with respect
to the marginal measure. For example, stationary MPR processes have this
property. On the other hand, as shown in \cite{SzabLev} the only MPR L\'{e}vy
process with orthogonal martingales is a Wiener process. Moreover, martingale
$M_{2}$ for L\'{e}vy MPR process is a reversed martingale for only a very
specific marginal distribution. Hence, only a family of polynomial martingales
does exist for all MPR processes.
\end{remark}

\section{Harnesses related classes of Markov processes}

Now let us analyze conditions for the process $\mathbf{X}$ to be a harness.
This subclass of Markov processes has been introduced by Hammersley in
\cite{Ham67}. Some basic properties of them under additional simplifying
assumptions were discussed in e.g. \cite{BryMaWe07}, \cite{SzablPoly}. In both
papers it was assumed that the family of martingales $\left\{  M_{n}\right\}
$ constitutes also a family of orthogonal polynomials. Below we do not assume
this. We start only with the assumption that an analyzed process is a MPR-process.

Recall that a stochastic process is a harness iff
\[
\forall s<t<u:\mathbb{E}(X_{t}|\mathcal{F}_{s,u})\allowbreak=\allowbreak
\hat{a}(s,t,u)X_{s}+\hat{b}(s,t,u)X_{u}+c(s,t,u),
\]
for some $\hat{a}$,$\hat{b}$,$c$ being functions of $s$,$t$,$u$. Notice that
taking expectations of both sides we see that
\begin{equation}
c(x,t,u)\allowbreak=\allowbreak EX_{t}\allowbreak-\allowbreak\hat{a}%
\mathbb{E}X_{s}\allowbreak-\allowbreak\hat{b}\mathbb{E}X_{u}. \label{cc}%
\end{equation}

\subsection{Harnesses\label{HAR}}

Since $M_{1}(s)$ is linear function of $X_{s}$ we can assume that
$M_{1}(t)\allowbreak=\allowbreak A_{t}(X_{t}-\mathbb{E}X_{t})$ for some
nonzero $A_{t}$, since we have assumed that $\mathbb{E}M_{1}(s)\allowbreak
=\allowbreak0$. Further, $M_{1}$ is a martingale and thus we have:%
\begin{equation}
A_{t}\mathbb{E}(X_{t}-\mathbb{E}X_{t}|\mathcal{F}_{\leq s})\allowbreak
=\allowbreak A_{s}(X_{s}-EX_{s})\text{.} \label{mm}%
\end{equation}
Now taking into account (\ref{cc}) and (\ref{mm}) we deduce that one can
define harnesses in an equivalent way, namely:%
\begin{equation}
\forall s<t<u:\mathbb{E}(M_{1}(t)|\mathcal{F}_{s,u})\allowbreak=\allowbreak
a(s,t,u)M_{1}(s)+b(s,t,u)M_{1}(u), \label{har1}%
\end{equation}
almost surely with $a\allowbreak=\allowbreak A_{t}\hat{a}/A_{s}$ and
$b\allowbreak=\allowbreak A_{t}\hat{b}/A_{u}$.

Let us multiply both sides of (\ref{har1}) by $M_{1}(s)$. We have then,
$m_{1}(s)\allowbreak=\allowbreak am_{1}(s)+bm_{1}(s)$. Since the function
$m_{1}(s)$ is positive and increasing we deduce that
$a(s,t,u)+b(s,t,u)\allowbreak=\allowbreak1$.

Now let us multiply both sides of (\ref{har1}) by $M_{1}(u)$. We will get
then, $m_{1}(t)\allowbreak=\allowbreak am_{1}(s)+bm_{1}(u)$. So
\[
a(s,t,u)=\frac{m_{1}(u)-m_{1}(t)}{m_{1}(u)-m_{1}(s)},~~b(s,t,u)=\frac
{m_{1}(t)-m_{1}(s)}{m_{1}(u)-m_{1}(s)}.
\]

We have the following observation.

\begin{theorem}
\label{har}If process $\mathbf{X}$ is a harness then, \newline1. $\left\{
M_{1}(s)/m_{1}(s)\right\}  $ is a reversed martingale, \newline2. $\forall
n\geq1$ there exist constants $a_{j,n}$ , $b_{j,n}$, $j\allowbreak
=\allowbreak1$,$\ldots$,$n+1$ such that:
\begin{equation}
M_{1}(t)M_{n}(t)\allowbreak=\allowbreak\sum_{j=1}^{n+1}(\alpha_{j,n}%
m_{1}(t)+\beta_{j,n})M_{j}(t)+\chi_{n,1}m_{1}(t).\label{struc}%
\end{equation}

\end{theorem}

\begin{proof}
Shifted to Section \ref{dow}.
\end{proof}

As a corollary we get the following result.

\begin{theorem}
\label{ortog}If all polynomials $M_{i}(t)$ are orthogonal, then condition
(\ref{struc}) is also sufficient for the process $\mathbf{X}$ to be a harness.
\end{theorem}

\begin{proof}
Shifted to Section \ref{dow}.
\end{proof}

\begin{remark}
Notice that condition (\ref{struc}) is in fact a generalization of necessary
conditions that the 3-term recurrence of the family of orthogonal polynomial
martingales have to satisfy so as the process this family defines to be a
harness (compare results of \cite{SzablPoly}). Here we do not have orthogonal
martingales hence, there are no 3-term recurrences. However the fact that the
coefficients of the $M_{1}M_{n}$ expansion are linear functions of $m_{1}$
remains true.
\end{remark}

\subsection{Quadratic harnesses\label{QHAR}}

The class of quadratic harnesses (QH) has been intensively studied in recent
years by W. Bryc, J. Weso\l owski and occasionally by W. Matysiak in several
papers \cite{brwe05}, \cite{BryMaWe07}, \cite{BryWe}, \cite{BryBo},
\cite{BryMaWe}, \cite{BryWe10} under some more or less restricting and
regularizing assumptions of which the most important was the postulated
existence of the family of orthogonal polynomial martingales of the specific
type with linear dependence on $t$ of the 3-term recurrence coefficients. In
\cite{SzablPoly} Szab\l owski studied quadratic harnesses under the less
restricting assumption of existence of orthogonal polynomial martingales and
the additional assumption that the transitional distribution is absolutely
continuous with respect to the marginal distributions and more over that the
Radon--Nikodym derivative of these distributions is square integrable with
respect to the marginal one. Of course, generalization of the results of Bryc
\& Weso\l owski were obtained. In the present paper we study existence of
quadratic harnesses under no restrictions except of course, from the
assumption that it belongs to MPR class.

Let us recall that in \cite{SzablPoly} the following, slightly more general
than in the works of Bryc and Weso\l owski, definition of a QH was used.

The process $\mathbf{X}$ is QH if it is a harness and almost surely
\begin{equation}
\mathbb{E}(M_{2}(t)|\mathcal{F}_{s,u})=AM_{2}(s)+BM_{1}(s)M_{1}(u)+CM_{2}%
(u)+DM_{1}(s)+EM_{1}(u)+F, \label{QH}%
\end{equation}
for some functions $A$, $B$, $C$, $D$, $E$, $F$ of $s$,$t$,$u$. Immediately
notice that since $\mathbb{E}(M_{i}(s))\allowbreak=\allowbreak0$, $i=1$,$2$,
we have: $Bm_{1}(s)\allowbreak+\allowbreak F\allowbreak=\allowbreak0$.

To proceed further, let us assume in the sequel that the martingales $M_{1}$
and $M_{2}$ are orthogonal i.e. $\forall t\in I:\mathbb{E}(M_{1}%
(t)M_{2}(t))\allowbreak=\allowbreak0$.

\begin{lemma}
Suppose that a MPR process $\mathbf{X}$ is a harness then, keeping in mind the
following definitions of parameters $\alpha_{i,j}$ $\beta_{i,j}$, $i=1$%
,$2$,$3$, $j=1$,$2$ and $\chi_{3,1}$ that follow assumption that $\mathbf{X}$
is a harness:%
\begin{equation}
M_{1}(t)^{2}\allowbreak=\allowbreak(\alpha_{2,1}m_{1}(t)+\beta_{2,1}%
)M_{2}(t)\allowbreak+\allowbreak(\alpha_{1,1}m_{1}(t)+\beta_{1,1}%
)M_{1}(t)+m_{1}(t), \label{p1}%
\end{equation}%
\begin{align}
M_{1}(t)M_{2}(t)\allowbreak &  =\allowbreak(\alpha_{3,2}m_{1}(t)+\beta
_{3,2})M_{3}(t)+(\alpha_{2,2}m_{1}(t)+\beta_{2,2})M_{2}(t)\label{p2}\\
&  +(\alpha_{1,2}m_{1}(t)+\beta_{1,2})M_{1}(t),\nonumber
\end{align}
\begin{equation}
\mathbb{E}M_{3}(t)M_{1}(t)\allowbreak=\allowbreak\chi_{3,1}m_{1}(t),
\label{p3}%
\end{equation}
and denoting for simplicity $\hat{a}\allowbreak=\allowbreak\alpha_{3,2}%
\chi_{3,1}+\alpha_{1,2}$, $a\allowbreak=\allowbreak\beta_{3,2}\chi
_{3,1}\allowbreak+\allowbreak\beta_{1,2}$, we have:

1.%
\begin{align*}
\mathbb{E}(M_{1}^{2}(t))\allowbreak &  =\allowbreak m_{1}(t),~\mathbb{E}%
(M_{1}(t)M_{2}(t))\allowbreak=\allowbreak0,\\
\mathbb{E}(M_{1}^{3}(t)\allowbreak &  =\allowbreak(\alpha_{1,1}m_{1}%
(t)\allowbreak+\allowbreak\beta_{1,1})\allowbreak m_{1}(t),\\
\mathbb{E}(M_{1}^{2}(s)M_{2}(s))\allowbreak &  =\allowbreak m_{1}(s)(\hat
{a}m_{1}(s)+a),
\end{align*}

2.
\begin{align*}
m_{2}(t)\allowbreak &  =\allowbreak\frac{(\hat{a}m_{1}(t)+a)}{(\alpha
_{2,1}m_{1}(t)+\beta_{2,1})}m_{1}(t)\allowbreak,\\
\mathbb{E}(M_{1}^{2}(t)M_{1}^{2}(s))\allowbreak &  =\allowbreak m_{1}%
(s)((\alpha_{2,1}m_{1}(t)+\beta_{2,1})(\hat{a}m_{1}(s)+a)\allowbreak\\
&  +\allowbreak(\alpha_{1,1}m_{1}(t)+\beta_{1,1})(\alpha_{1,1}m_{1}%
(s)+\beta_{1,1})+\allowbreak\allowbreak m_{1}(t)).
\end{align*}

\end{lemma}

\begin{proof}
Keeping in mind that $\mathbb{E}M_{i}(t)\allowbreak=\allowbreak0$ and our
assumption is valid and looking also on expansions (\ref{p1})-(\ref{p3}) we
see that 1. is trivially true. To see that the first statement of 2. is true
we compute the $\mathbb{E}(M_{2}(t)M_{1}^{2}(t))$ in two ways. Firstly by
multiplying (\ref{p1}) by $M_{2}(t)$ and taking expectation, secondly by
multiplying (\ref{p2}) by $M_{1}(t)$ and taking expectation.
\end{proof}

\begin{proposition}
\label{QR}Suppose that MPR process $\mathbf{X}$ is a quadratic harness then,
functions $A$, $B$, $C$, $D$, $E$ $F$ defined by (\ref{QH}) are given by the
following formulae.:%
\begin{gather}
A\allowbreak=h\frac{(a_{21}m(s)+b_{21})(b_{21}a\kappa-b_{21}\hat{a}%
(\lambda-\kappa)m(t)+b_{21}\hat{a}\kappa m(u)+a_{12}\hat{a}\kappa
m(u)m(t))}{((a_{21}m(t)+b_{21})(b_{21}a\kappa-b_{21}\hat{a}(\lambda
-\kappa)m(s)+b_{21}\hat{a}\kappa m(u)+a_{12}\hat{a}\kappa m(u)m(t))}%
,\label{Aa}\\
B=\frac{(m(t)-m(s))\lambda b_{21}\hat{a}}{(a_{21}m(t)+b_{21})(b_{21}%
a\kappa-b_{21}\hat{a}(\lambda-\kappa)m(s)+b_{21}\hat{a}\kappa m(u)+a_{12}%
\hat{a}\kappa m(u)m(t))},\label{Bb}\\
C=(1-h)\frac{(a_{21}m(u)+b_{21})(b_{21}a\kappa-b_{21}\hat{a}(\lambda
-\kappa)m(s)+b_{21}\hat{a}\kappa m(t)+a_{12}\hat{a}\kappa m(s)m(t))}%
{(a_{21}m(t)+b_{21})(b_{21}a\kappa-b_{21}\hat{a}(\lambda-\kappa)m(s)+b_{21}%
\hat{a}\kappa m(u)+a_{12}\hat{a}\kappa m(u)m(t))},\label{Cc}\\
D=-b_{11}B,~~E=-a_{11}Bm(s),~~F=-Bm(s), \label{Dd}%
\end{gather}
where we denoted
\begin{align*}
m(.)\allowbreak &  =\allowbreak m_{1}(.),~\hat{a}\allowbreak=\allowbreak
\alpha_{3,2}\chi_{3,1}+\alpha_{1,2},\\
a\allowbreak &  =\allowbreak\beta_{3,2}\chi_{3,1}\allowbreak+\allowbreak
\beta_{1,2},~\kappa\allowbreak=\allowbreak(1+a_{11}b_{11}+a_{21}a),\\
\lambda\allowbreak &  =\allowbreak(b_{21}\hat{a}-a_{21}a),
\end{align*}
and
\[
h=h(s,t,u)\allowbreak=\allowbreak\frac{m(u)-m(t)}{m(u)-m(s)}.
\]

\end{proposition}

\begin{proof}
Shifted to Section \ref{dow}.
\end{proof}

As a corollary we obtain the following nice result.

\begin{theorem}
\label{rev}If $\mathbf{X}$ is a quadratic harness then, $M_{2}$ is a reversed martingale.
\end{theorem}

\begin{proof}
Shifted to Section \ref{dow}.
\end{proof}

\section{Proofs\label{dow}}

\begin{proof}
[Proof of the Theorem \ref{har}]$1$. Let us multiply both sides of
(\ref{har1}) by $M_{n}(u)$. We get then,
\[
\mathbb{E}M_{1}(t)M_{n}(t)\allowbreak=\allowbreak\frac{m_{1}(u)-m_{1}%
(t)}{m_{1}(u)-m_{1}(s)}\mathbb{E}M_{1}(s)M_{n}(s)+\frac{m_{1}(t)-m_{1}%
(s)}{m_{1}(u)-m_{1}(s)}\mathbb{E}M_{1}(u)M_{n}(u).
\]
For simplicity, let us denote $h_{n}(t)\allowbreak=\allowbreak\mathbb{E}%
M_{1}(t)M_{n}(t)$. Notice that the above mentioned equality is equivalent to
to the following:%
\[
h_{n}(t)-h_{n}(s)\allowbreak=\allowbreak\frac{m_{1}(t)-m_{1}(s)}%
{m_{1}(u)-m_{1}(s)}(h_{n}(u)-h_{n}(s),
\]
for $s<t<u$. Thus, we deduce that $\frac{h_{n}(t)-h_{n}(s)}{m_{1}(t)-m_{1}%
(s)}$ does not depend on $t$. Consequently, $h_{n}(t)\allowbreak=\allowbreak
h_{n}(s)\allowbreak+\allowbreak C_{n}(s)(m_{1}(t)-m_{1}(s))$ for some constant
$C_{n}(s)$ for all $s<t<u$. Besides, by our assumptions $h_{n}(0)\allowbreak
=\allowbreak0$ and $m_{1}(0)\allowbreak=\allowbreak0$. Hence, we deduce that
$h_{n}(s)\allowbreak=\allowbreak\beta_{n}m_{1}(s)$. By Proposition \ref{wst}
this yields that $M_{1}(s)/m_{1}(s)$ is a reversed martingale.

2. Since $M_{n}(t)$ are polynomials we deduce that there exist $n+1$
continuous functions $\left\{  \delta_{j,n}(t)\right\}  _{j=0}^{n+1}$ such
that
\begin{equation}
M_{1}(t)M_{n}(t)\allowbreak=\allowbreak\sum_{j=0}^{n+1}\delta_{j,n}%
(t)M_{j}(t). \label{pom1}%
\end{equation}
Since $\mathbf{X}$ is a harness, we deduce by the previous theorem that
$\delta_{0,n}(t)\allowbreak=\allowbreak\chi_{n,1}m_{1}(t)$. To show that
$\delta_{j,n}(t)\allowbreak=\allowbreak\alpha_{j,n}m_{1}(t)\allowbreak
+\allowbreak\beta_{j,n}$ for some reals $\alpha_{j,n}$, $\beta_{j,n}$ let us
multiply (\ref{har1}) by $M_{n}(u)$ and then, calculate conditional
expectation with respect to $\mathcal{F}_{\leq s}$. We get then:%
\[
\mathbb{E(}M_{1}(t)M_{n}(t)|\mathcal{F}_{\leq s})\allowbreak=\allowbreak
a(s,t,u)M_{1}(s)M_{n}(s)+b(s,t,u)\mathbb{E(}M_{1}(u)M_{n}(u)|\mathcal{F}_{\leq
s}).
\]
Now apply (\ref{pom1}) and use the fact that $M_{j}(t)$, $j\allowbreak
=\allowbreak1,\ldots,n+1$ are martingales, thus:%

\begin{align*}
\sum_{j=1}^{n+1}\delta_{j,n}(t)M_{j}(s)\allowbreak+\beta_{0,n}m_{1}(t)  &
=\allowbreak a(s,t,u)\sum_{j=1}^{n+1}\delta_{j,n}(s)M_{j}(s)\allowbreak
+\chi_{n,1}m_{1}(s)\\
&  +b(s,t,u)\sum_{j=1}^{n+1}\delta_{j,n}(u)M_{j}(s)\allowbreak+\chi_{n,1}%
m_{1}(u).
\end{align*}
Since polynomials $M_{i}$ are linearly independent we deduce that the
functions $\delta_{j,n}$ have to satisfy the following $n+1$ equations:%
\[
\delta_{j,n}(t)\allowbreak=\allowbreak\frac{m_{1}(u)-m_{1}(t)}{m_{1}%
(u)-m_{1}(s)}\delta_{j,n}(s)+\frac{m_{1}(t)-m_{1}(s)}{m_{1}(u)-m_{1}(s)}%
\delta_{j,n}(u),
\]
for $j\allowbreak=\allowbreak1$,$\ldots$,$n+1$. Subtracting $\delta_{j,n}(s)$
from both sides of the above equality we get:
\[
\frac{\delta_{j,n}(t)-\delta_{j,n}(s)}{m_{1}(t)-m_{1}(s)}\allowbreak
=\allowbreak\frac{\delta_{j,n}(u)-\delta_{j,n}(s)}{m_{1}(u)-m_{1}(s)}.
\]
We deduce that $\frac{\delta_{j,n}(t)-\delta_{j,n}(s)}{m_{1}(t)-m_{1}(s)}$
does not depend on $t$ so $\delta_{j,n}(t)\allowbreak=\allowbreak\delta
_{j,n}(s)+f_{j,n}(s)(m_{1}(t)-m_{1}(s))$ for some $f_{j,n}(s)$. Now notice
that for $s=\allowbreak0$ we have $m_{1}(0)\allowbreak=\allowbreak
\allowbreak0$ hence, $\delta_{j,n}(t)\allowbreak=\allowbreak\beta
_{j,n}\allowbreak+\allowbreak\alpha_{j,n}m_{1}(t)$, where we denoted
$\delta_{j,n}(0)\allowbreak=\allowbreak\beta_{j,n}$ and $f_{j,n}%
(0)\allowbreak=\allowbreak\alpha_{j,n}$.
\end{proof}

\begin{proof}
[Proof of Theorem \ref{ortog}]First of all notice that since polynomials
$\left\{  M_{i}\right\}  $ are assumed to be orthogonal, condition
(\ref{struc}) takes a for of 3-term recurrence%
\begin{gather*}
M_{1}(t)M_{n}(t)\allowbreak=\allowbreak(\hat{\alpha}_{n+1}m_{1}(t)+\alpha
_{n+1})M_{n+1}(t)+\\
(\hat{\beta}_{n}m_{1}(t)+\beta_{n})M_{n}(t)+(\hat{\gamma}_{n-1}m_{1}%
(t)+\gamma_{n-1})M_{n-1}(t),
\end{gather*}
for $n\geq2$. Recall that also following Theorem \ref{ort}, all martingales
$M_{n}(s)$ multiplied by $1/m_{n}(s)$ are also reversed martingales. If
$n\allowbreak=\allowbreak1$ we have
\[
M_{1}^{2}(t)\allowbreak=\allowbreak(\hat{\alpha}_{2}m_{1}(t)+\alpha_{2}%
)M_{2}(t)+(\hat{\beta}_{1}m_{1}(t)+\beta_{1})M_{1}(t)+m_{1}(t).
\]
Hence, we deduce that $\hat{\alpha}_{1}\allowbreak=\allowbreak\gamma
_{0}\allowbreak=\allowbreak0$, $\alpha_{1}\allowbreak=\allowbreak\hat{\gamma
}_{0}\allowbreak=\allowbreak1$. Secondly notice that for $\sigma\leq
s<t<u\leq\upsilon$%
\begin{gather*}
\mathbb{E}(\mathbb{E}(M_{n}(\sigma)M_{1}(t)M_{k}(\upsilon)|\mathcal{F}%
_{s,u})\allowbreak)=\allowbreak\frac{m_{n}(\sigma)}{m_{n}(t)}\mathbb{E}%
(M_{n}(t)M_{1}(t)M_{k}(t))\\
=\frac{m_{n}(\sigma)}{m_{n}(t)}\mathbb{E}(M_{n}(t)((\hat{\alpha}_{k+1}%
m_{1}(t)+\alpha_{k+1})M_{k+1}(t)\\
+(\hat{\beta}_{k}m_{1}(t)+\beta_{k})M_{k}(t)+(\hat{\gamma}_{k-1}%
m_{1}(t)+\gamma_{k-1})M_{k-1}(t))\\
=\left\{
\begin{array}
[c]{ccc}%
m_{n}(\sigma)(\hat{\alpha}_{n}m_{1}(t)+\alpha_{n}) & if & n\allowbreak
=\allowbreak k+1\\
m_{n}(\sigma)(\hat{\beta}_{n}m_{1}(t)+\beta_{n}) & if & n=k\\
m_{n}(\sigma)(\hat{\gamma}_{n}m_{1}(t)+\gamma_{n}) & if & n=k-1\\
0 & if & n\notin\{k-1,k,k+1\}
\end{array}
.\right.
\end{gather*}
On the other hand, by similar reasoning
\begin{align*}
&  \mathbb{E}(\mathbb{E}(M_{n}(\sigma)M_{1}(s)M_{k}(\upsilon)|\mathcal{F}%
_{s,u})\allowbreak)\allowbreak\\
&  =\allowbreak\left\{
\begin{array}
[c]{ccc}%
m_{n}(\sigma)(\hat{\alpha}_{n}m_{1}(s)+\alpha_{n}) & if & n\allowbreak
=\allowbreak k+1\\
m_{n}(\sigma)(\hat{\beta}_{n}m_{1}(s)+\beta_{n}) & if & n=k\\
m_{n}(\sigma)(\hat{\gamma}_{n}m_{1}(s)+\gamma_{n}) & if & n=k-1\\
0 & if & n\notin\{k-1,k,k+1\}
\end{array}
,\right.
\end{align*}
and%
\[
\mathbb{E}(\mathbb{E}(M_{n}(\sigma)M_{1}(u)M_{k}(\upsilon)|\mathcal{F}%
_{s,u})\allowbreak)\allowbreak\allowbreak=\allowbreak\allowbreak\left\{
\begin{array}
[c]{ccc}%
m_{n}(\sigma)(\hat{\alpha}_{n}m_{1}(u)+\alpha_{n}) & if & n\allowbreak
=\allowbreak k+1\\
m_{n}(\sigma)(\hat{\beta}_{n}m_{1}(u)+\beta_{n}) & if & n=k\\
m_{n}(\sigma)(\hat{\gamma}_{n}m_{1}(u)+\gamma_{n}) & if & n=k-1\\
0 & if & n\notin\{k-1,k,k+1\}
\end{array}
.\right.
\]
Consequently, we see that we have for any $n$,$m\in\mathbb{N}_{0}$%
\begin{gather*}
\mathbb{E}(M_{n}(\sigma)M_{1}(s)M_{k}(\upsilon)|\mathcal{F}_{s,u})=\\
a(s,t,u)\mathbb{E}(M_{n}(\sigma)M_{1}(s)M_{k}(\upsilon)|\mathcal{F}%
_{s,u})+b(s,t,u)\mathbb{E}(M_{n}(\sigma)M_{1}(u)M_{k}(\upsilon)|\mathcal{F}%
_{s,u}).
\end{gather*}
Now we have to refer to way of reasoning used in the proof of Proposition
\ref{wst}. Let us consider $\sigma_{1}<\sigma_{2}\leq s<t<u\leq\upsilon
_{2}<\upsilon_{1}$. Then,
\begin{gather*}
\mathbb{E}(M_{n_{1}}(\sigma_{1})M_{n_{2}}(\sigma_{2})M_{1}(t)M_{k_{2}%
}(\upsilon_{2})M_{k_{1}}(\upsilon_{1}))\allowbreak=\allowbreak\\
\frac{m_{n_{1}}(\sigma_{1})}{m_{n_{1}}(\sigma_{2})}\mathbb{E((}\sum
_{j=0}^{n_{1}+n_{2}}\phi_{j,n_{1},n_{2}}(\sigma_{2})M_{j}(\sigma_{2}%
))M_{1}(t)\sum_{j=0}^{m_{1}+m_{2}}\phi_{j,m_{1},m_{2}}(\upsilon_{1}%
)M_{j}(t))\\
=\frac{m_{n_{1}}(\sigma_{1})}{m_{n_{1}}(\sigma_{2})}\sum_{j=0}^{n_{1}+n_{2}%
}\phi_{j,n_{1},n_{2}}(\sigma_{2})\frac{m_{j}(\sigma_{2})}{m_{j}(t)}\sum
_{i=0}^{m_{1}+m_{2}}\phi_{j,m_{1},m_{2}}(\upsilon_{1})\mathbb{E}(M_{j}%
(t)M_{1}(t)M_{i}(t))\\
=\frac{m_{n_{1}}(\sigma_{1})}{m_{n_{1}}(\sigma_{2})}\sum_{j=0}^{n_{1}+n_{2}%
}\phi_{j,n_{1},n_{2}}(\sigma_{2})m_{j}(\sigma_{2})\sum_{i=0}^{m_{1}+m_{2}}%
\phi_{j,m_{1},m_{2}}(\upsilon_{1})\\
\times\left\{
\begin{array}
[c]{ccc}%
\hat{\alpha}_{n}m_{1}(t)+\alpha_{n} & if & j=i+1\\
\hat{\beta}_{n}m_{1}(t)+\beta_{n} & if & j=i\\
\hat{\gamma}_{n}m_{1}(t)+\gamma_{n} & if & j=i-1\\
0 & if & j\notin\{i-1,i,i+1\}
\end{array}
\right.  .
\end{gather*}
Similarly, for $\mathbb{E}(M_{n_{1}}(\sigma_{1})M_{n_{2}}(\sigma_{2}%
)M_{1}(s)M_{k_{2}}(\upsilon_{2})M_{k_{1}}(\upsilon_{1}))\allowbreak$ and
\newline$\mathbb{E}(M_{n_{1}}(\sigma_{1})M_{n_{2}}(\sigma_{2})M_{1}%
(u)M_{k_{2}}(\upsilon_{2})M_{k_{1}}(\upsilon_{1}))$. Using a previous result
we see that also in this case equation (\ref{har1}) is satisfied and so on for
any finite products of the form $\mathbb{E}(\prod_{j=1}^{l}M_{n_{j}}%
(\sigma_{j})M_{1}(t)\prod_{j=1}^{\lambda}M_{m_{j}}(\upsilon_{j})$ for
$\sigma_{1}<\ldots\sigma_{l}\leq s<t<u\leq\upsilon_{\lambda}<\ldots
<\upsilon_{1}$. Finally we utilize the fact that linear combinations of
functions $\prod_{j=1}^{l}M_{n_{j}}(\sigma_{j})\prod_{j=1}^{\lambda}M_{m_{j}%
}(\upsilon_{j})$ are dense in the space of square integrable functions
measurable w.r. to $\mathcal{F}_{s,u}$.
\end{proof}

\begin{proof}
[Proof of Proposition \ref{QR}]First we multiply both sides of (\ref{QH})
successively by $M_{1}(s)$, $M_{1}(u)$, $M_{1}(s)M_{1}(u)$ , $M_{2}(s)$,
$M_{2}(u)$ and then, we take expectation. We get that:
\begin{gather*}
\mathbb{E}M_{2}(s)M_{1}(s)\allowbreak=\allowbreak A\mathbb{E}M_{2}%
(s)M_{1}(s)+B\mathbb{E}M_{1}^{3}(s)+C\mathbb{E}M_{2}(s)M_{1}(s)+(D+E)m_{1}%
(s),\\
\mathbb{E}M_{2}(t)M_{1}(t)=A\mathbb{E}M_{2}(s)M_{1}(s)+B\frac{m_{1}(s)}%
{m_{1}(u)}\mathbb{E}M_{1}^{3}(u)\\
+C\mathbb{E}M_{2}(u)M_{1}(u)+Dm_{1}(s)+Em_{1}(u),\\
\frac{m_{1}(s)}{m_{1}(t)}\mathbb{E}M_{1}^{2}(t)M_{2}(t)=A\mathbb{E}M_{1}%
^{2}(s)M_{2}(s)+B\mathbb{E}M_{1}^{2}(s)M_{1}^{2}(u)\\
+\frac{m_{1}(s)}{m_{1}(u)}C\mathbb{E}M_{1}^{2}(u)M_{2}(u)+D\mathbb{E}M_{1}%
^{3}(s)+E\frac{m_{1}(s)}{m_{1}(u)}\mathbb{E}M_{1}^{3}(u)-Bm_{1}^{2}(s),\\
m_{2}(s)\allowbreak=\allowbreak(A+C)m_{2}(s)+B\mathbb{E}M_{1}^{2}%
(s)M_{2}(s)+(D+E)\mathbb{E}M_{1}(s)M_{2}(s),\\
m_{2}(t)=Am_{2}(s)+Cm_{2}(u)+\frac{m_{1}(s)}{m_{1}(u)}B\mathbb{E}M_{1}%
^{2}(u)M_{2}(u)+(\frac{m_{1}(s)}{m_{1}(u)}D+E)\mathbb{E}M_{1}(u)M_{2}(u).
\end{gather*}
Using our assumption we get.%
\begin{align*}
0  &  =B(\alpha_{1,1}m_{1}(s)\allowbreak+\allowbreak\beta_{1,1})+D+E,\\
0  &  =Bm_{1}(s)(\alpha_{1,1}m_{1}(u)\allowbreak+\allowbreak\beta
_{1,1})+Dm_{1}(s)+Em_{1}(u),
\end{align*}
from which it follows that $D\allowbreak=\allowbreak-\beta_{1,1}B$ and
$E\allowbreak=\allowbreak-\alpha_{1,1}m_{1}(s)B$ (to compare (\ref{Dd})),
\[
m_{2}(t)\allowbreak=\allowbreak\frac{(\alpha_{3,2}m_{1}(t)+\beta_{3,2}%
)}{(\alpha_{2,1}m_{1}(t)+\beta_{2,1})}\chi_{3,1}m_{1}(t)\allowbreak
+\frac{(\alpha_{1,2}m_{1}(t)+\beta_{1,2})}{(\alpha_{2,1}m_{1}(t)+\beta_{2,1}%
)}m_{1}(t).
\]
Now we utilize our assumptions and cancel out $m_{1}(s)$ in third, fourth and
fifth equation:
\begin{gather*}
((\alpha_{3,2}m_{1}(t)+\beta_{3,2})\chi_{3,1}+(\alpha_{1,2}m_{1}%
(t)+\beta_{1,2}))\\
=A((\alpha_{3,2}m_{1}(s)+\beta_{3,2})\chi_{3,1}+(\alpha_{1,2}m_{1}%
(s)+\beta_{1,2}))\\
+B((\alpha_{2,1}m_{1}(u)+\beta_{2,1})((\alpha_{3,2}m_{1}(s)+\beta_{3,2}%
)\chi_{3,1}+\\
(\alpha_{1,2}m_{1}(s)+\beta_{1,2}))+(\alpha_{1,1}m_{1}(u)+\beta_{1,1}%
)(\alpha_{1,1}m_{1}(s)+\beta_{1,1})\allowbreak+\allowbreak m_{1}(u))\\
+C(\alpha_{3,2}m_{1}(u)+\beta_{3,2})\chi_{3,1}+(\alpha_{1,2}m_{1}%
(u)+\beta_{1,2})\\
-\beta_{1,1}B(\alpha_{1,1}m_{1}(s)\allowbreak+\allowbreak\beta_{1,1}%
)-\alpha_{1,1}m_{1}(s)B(\alpha_{1,1}m_{1}(u)\allowbreak+\allowbreak\beta
_{1,1})-Bm_{1}(s),\\
(1-A-C)=B(\alpha_{2,1}m_{1}(s)+\beta_{2,1}),\\
m_{1}(t)(\frac{(\alpha_{3,2}m_{1}(t)+\beta_{3,2})}{(\alpha_{2,1}m_{1}%
(t)+\beta_{2,1})}\chi_{3,1}\allowbreak+\allowbreak\frac{(\alpha_{1,2}%
m_{1}(t)+\beta_{1,2})}{(\alpha_{2,1}m_{1}(t)+\beta_{2,1})})=\\
Am_{1}(s)(\frac{(\alpha_{3,2}m_{1}(s)+\beta_{3,2})}{(\alpha_{2,1}%
m_{1}(s)+\beta_{2,1})}\chi_{3,1}\allowbreak+\allowbreak\frac{(\alpha
_{1,2}m_{1}(s)+\beta_{1,2})}{(\alpha_{2,1}m_{1}(s)+\beta_{2,1})})\\
+Cm_{1}(u)(\frac{(\alpha_{3,2}m_{1}(u)+\beta_{3,2})}{(\alpha_{2,1}%
m_{1}(u)+\beta_{2,1})}\chi_{3,1}\allowbreak\\
+\allowbreak\frac{(\alpha_{1,2}m_{1}(u)+\beta_{1,2})}{(\alpha_{2,1}%
m_{1}(u)+\beta_{2,1})})+m_{1}(s)B((\alpha_{3,2}m_{1}(u)+\beta_{3,2})\chi
_{3,1}+(\alpha_{1,2}m_{1}(u)+\beta_{1,2})).
\end{gather*}

Solving these these 3 equations, we obtain (\ref{Aa})-(\ref{Cc}).
\end{proof}

\begin{proof}
[Proof of Theorem \ref{rev}.]Set $n\geq3$ and let us multiply (\ref{QH}) by
$M_{n}(u)$ and integrate. We get then,%
\begin{gather*}
\mathbb{E}(M_{2}(t)M_{n}(t)\allowbreak=\allowbreak A\mathbb{E}(M_{2}%
(s)M_{n}(s))+\frac{m_{1}(s)}{m_{1}(u)}B\mathbb{E}(M_{1}^{2}(u)M_{n}%
(u)+C\mathbb{E}(M_{2}(u)M_{n}(u))\\
+D\chi_{n,1}m_{1}(s)+E\chi_{n,1}m_{1}(u).
\end{gather*}
Let us denote for simplicity $h_{n}(t)\allowbreak=\allowbreak\mathbb{E}%
(M_{2}(t)M_{n}(t))$. Using (\ref{p1}) we get%
\begin{gather*}
h_{n}(t)\allowbreak=\allowbreak Ah_{n}(s)+Ch_{n}(u)+\frac{m_{1}(s)}{m_{1}%
(u)}B(\alpha_{2,1}m_{1}(u)+\beta_{2,1})h_{n}(u)\\
+\frac{m_{1}(s)}{m_{1}(u)}B(\alpha_{1,1}m_{1}(u)+\beta_{1,1})m_{1}%
(u)\chi_{n,1}+D\chi_{n,1}m_{1}(s)+E\chi_{n,1}m_{1}(u).
\end{gather*}
Now notice that taking into account (\ref{Bb}) and (\ref{Dd}) we have that
\[
\frac{m_{1}(s)}{m_{1}(u)}B(\alpha_{1,1}m_{1}(u)+\beta_{1,1})m_{1}(u)\chi
_{n,1}\allowbreak+\allowbreak D\chi_{n,1}m_{1}(s)\allowbreak+\allowbreak
E\chi_{n,1}m_{1}(u)\allowbreak=\allowbreak0.
\]
Hence, we have to prove that only solution of the equation:%
\[
h_{n}(t)\allowbreak=\allowbreak Ah_{n}(s)+Ch_{n}(u)+\frac{m_{1}(s)}{m_{1}%
(u)}B(\alpha_{2,1}m_{1}(u)+\beta_{2,1})h_{n}(u),
\]
among the continuous functions $h_{n}($.$)$ is $\chi_{n,2}m_{n}($,$)$. To
prove this notice that the following three identities hold for all $s<t<u$:%
\begin{gather*}
A+B(\alpha_{2,1}m_{1}(s)+\beta_{2,1})+C=1,\\
Am_{1}(s)\hat{a}+B(m_{1}(u)\kappa+m_{1}(s)(\lambda-\kappa+\alpha_{21}a)+\\
\alpha_{21}\hat{a}m_{1}(s)m_{1}(u))+C\hat{a}m_{1}(u)=\hat{a}m_{1}(t),\\
Am_{2}(s)+\frac{m_{1}(s)}{m_{1}(u)}Bm_{2}(u)(\alpha_{2,1}m_{1}(u)+\beta
_{2,1})+Cm_{2}(u)=m_{2}(t).
\end{gather*}
We treat these identities as a system of linear equations in $(A$,$B$,$C)$
with matrix:%
\[
Mm=\left[
\begin{array}
[c]{ccc}%
1 & (\alpha_{2,1}m_{1}(s)+\beta_{2,1}) & 1\\
m_{1}(s)\hat{a} & (m_{1}(u)\kappa+m_{1}(s)(\lambda-\kappa+\alpha_{21}%
a)+\alpha_{21}\hat{a}m_{1}(s)m_{1}(u)) & \hat{a}m_{1}(u)\\
m_{2}(s) & \frac{m_{1}(s)}{m_{1}(u)}m_{2}(u)(\alpha_{2,1}m_{1}(u)+\beta
_{2,1}) & m_{2}(u)
\end{array}
\right]
\]
whose determinant is equal to:%
\begin{align*}
&  -\alpha_{21}\hat{a}\kappa m_{1}(s)m_{1}(u)^{2}\allowbreak+\allowbreak
\alpha_{21}\hat{a}(\kappa+\mathcal{\alpha}_{21}(a-1))m_{1}(s)^{2}%
m_{1}(u)\allowbreak\\
&  -\allowbreak\beta_{21}\hat{a}(\lambda-\kappa-\mathcal{\alpha}%
_{21}(a-1))m_{1}(s)^{2}\allowbreak+\allowbreak\beta_{21}\hat{a}\kappa
m_{1}(u)^{2}\allowbreak+\allowbreak\beta_{12}\hat{a}(\beta_{21}\hat{a}%
-\alpha_{21})m_{1}(s)m_{1}(u)\allowbreak\\
&  -\allowbreak\beta_{21}a\kappa m_{1}(u)\allowbreak+\allowbreak\beta
_{21}a(\kappa+\alpha_{21}(a-1))m_{1}(s).
\end{align*}
This determinant is in fact a polynomial in $m_{1}(s)$ and $m_{1}(u)$ of order
$3$ with not all coefficients equal to zero and as such is equal to zero only
for $u$ and $s$ from a set of measure zero. Hence, there exists only one
solution of this system of equations which is known. Thus, if the third row of
the matrix was replaced by some functions $(l(s)$,$\frac{m_{1}(s)}{m_{1}%
(u)}l(u)(\alpha_{2,1}m_{1}(u)+\beta_{2,1})$,$l(u))$ with $l(s)\neq\chi
_{n,2}m_{2}(s)$ for some $\chi_{n,2}$, we would have got a different solution,
which is not the case. Hence, $h_{n}(s)\allowbreak=\allowbreak\chi_{n,2}%
m_{2}(s)$ which by Proposition \ref{wst}, implies that $M_{2}(s)$ is a
reversed martingale.
\end{proof}

\end{document}